\documentclass[12pt,reqno]{amsart}

\usepackage{amsmath,amssymb,amsfonts,amsthm,latexsym,graphicx,multirow,hyperref,enumerate,tikz}
\usetikzlibrary{calc}

\newcommand\A{\mathrm{A}}   \newcommand\Aut{\mathrm{Aut}}
 
\newcommand\C{\mathrm{C}} \newcommand\calB{\mathcal{B}} \newcommand\Cay{\mathrm{Cay}} \newcommand\Cen{\mathbf{C}}  
\newcommand\D{\mathrm{D}}  

 \newcommand\Fix{\mathrm{Fix}}

\newcommand\Inn{\mathrm{Inn}}

\newcommand\Nor{\mathbf{N}}

  \newcommand\PGL{\mathrm{PGL}}   \newcommand\POm{\mathrm{P\Omega}}  \newcommand\PSL{\mathrm{PSL}}   \newcommand\PSO{\mathrm{PSO}} \newcommand\PSp{\mathrm{PSp}} \newcommand\PSU{\mathrm{PSU}}

       \newcommand\Sy{\mathrm{S}}

\newtheorem{theorem}{Theorem}[section]
\newtheorem{lemma}[theorem]{Lemma}

\newtheorem{corollary}[theorem]{Corollary}

\newtheorem{conjecture}[theorem]{Conjecture}
\theoremstyle{definition}
\newtheorem{question}[theorem]{Question}

\newtheorem{notation}[theorem]{Notation}
\newtheorem{construction}[theorem]{Construction}

\begin{document}

\title[Graphical regular representations]{Graphical regular representations of $(2,p)$-generated groups}

\author[Xia]{Binzhou Xia}
\address{School of Mathematics and Statistics\\The University of Melbourne\\Parkville, VIC 3010\\Australia}
\email{binzhoux@unimelb.edu.au}


\maketitle

\begin{abstract}
For groups $G$ that can be generated by an involution and an element of odd prime order, this paper gives a sufficient condition for a certain Cayley graph of $G$ to be a graphical regular representation (GRR), that is, for the Cayley graph to have full automorphism group isomorphic to $G$. This condition enables one to show the existence of GRRs of prescribed valency for a large class of groups, and in this paper, $k$-valent GRRs of finite nonabelian simple groups with $k\geq5$ are considered.


\end{abstract}

\section{Introduction}

The problem of whether a given group can be represented as the automorphism group of a graph was considered at a very early stage of graph theory. K\"{o}nig conjectured in his 1936 book `Theorie der endlichen und unendlichen Graphen'~\cite{Konig1990}, the first textbook on the field of graph theory, that every finite group is the automorphism group of a finite graph. K\"{o}nig's conjecture was proved in 1939 by Frucht~\cite{Frucht1939}, who, in 1949, also proved a stronger version stating that every finite group is the automorphism group of a cubic graph~\cite{Frucht1949}. Later, in 1957, Sabidussi proved that for all integers $k\geq3$, every finite group is the automorphism group of a $k$-regular graph~\cite{Sabidussi1957}.

In Frucht's theorem, or in the more general Sabidussi's theorem, the graph whose automorphism group is the given group may not be vertex-transitive and may not have the same order as the group. A graph which satisfies both of these conditions is called a \emph{graphical regular representation (GRR)} of the group. In other words, a graph $\Gamma$ is called a GRR of a group $G$ if $\Aut(\Gamma)$ acts regularly on the vertex set of $\Gamma$ and is isomorphic to $G$. After considerable work by many authors, Godsil~\cite{Godsil1978} determined which finite groups have a GRR. However, at this stage, a Sabidussi-like theorem concerning GRRs of a prescribed valency is far out of reach. Even in the special case that the valency is $3$, although it has attracted attention from several authors over the last few decades~\cite{CFP1981,FLWX2002,Godsil1983,LXZZ2022,Spiga2018,Xia2020I,Xia2020II,XF2016,XZZ2022}, little is yet known of which groups have a cubic GRR.

Given a group $G$ and an inverse-closed subset $S$ of $G\setminus\{1\}$, the \emph{Cayley graph} $\Cay(G,S)$ of $G$ with \emph{connection set} $S$ is the graph with vertex set $G$ such that two vertices $x,y$ are adjacent if and only if $yx^{-1}\in S$. If one identifies $G$ with its right regular permutation representation, then $G$ is a subgroup of $\Aut(\Cay(G,S))$. Conversely, a graph whose automorphism group has a subgroup $G$ regular on the vertex set is isomorphic to a Cayley graph of $G$. Thus a GRR of a group $G$ is exactly a Cayley graph of $G$ whose automorphism group is isomorphic to $G$. It is clear from the definition that $\Cay(G,S)$ is connected if and only if $S$ generates $G$. Moreover, if $\Cay(G,S)$ is a GRR of $G$ then it is connected (see Lemma~\ref{lem5}). Denote
\[
\Aut(G,S)=\{\alpha\in\Aut(G)\mid S^\alpha=S\},
\]
the setwise stabilizer of $S$ in $\Aut(G)$. Godsil~\cite{Godsil1981} showed that
\begin{equation}\label{eq2}
\Nor_{\Aut(\Cay(G,S))}(G)=G\rtimes\Aut(G,S).
\end{equation}
Accordingly, $\Aut(G,S)=1$ is a necessary condition for $\Cay(G,S)$ to be a GRR of $G$. In some circumstances this condition also turns out to be sufficient, which makes it much easier to search for GRRs.
Due to the applications in proving the existence of GRRs as well as its own interest, it is natural to ask the following question:

\begin{question}\label{que1}
In what circumstances is $\Aut(G,S)=1$ a sufficient condition for $\Cay(G,S)$ to be a GRR?
\end{question}

A group is said to be \emph{$(a,b)$-generated} if it can be generated by two elements of order $a$ and $b$ respectively.
A partial answer to Question~\ref{que1} is given by Godsil in~\cite{Godsil1983} for $(2,p)$-generated groups $G$ with certain connection sets $S$ of size $3$, where $p$ is an odd prime. As an application, Godsil proved in the same paper that there exists a cubic GRR of the alternating group $\A_n$ and the symmetric group $\Sy_n$, respectively, for every $n\geq19$. This result initiated the particular interest in Question~\ref{que1} on finite simple groups $G$ when $|S|=3$, see~\cite{FLWX2002,Spiga2018}. As a problem related to Question~\ref{que1}, it is posed in~\cite[Problem~A]{FLWX2002} to determine the groups $G$ such that $\Aut(G,S)=1$ is a sufficient condition for all $S$ to make $\Cay(G,S)$ a GRR. Some partial results on this problem have been obtained for $p$-groups, see~\cite{Godsil1981,LS2000}.

Our first theorem in this paper addresses Question~\ref{que1} for certain $G$ and $S$, where $G$ is $(2,p)$-generated for some prime $p$ and $|S|\geq5$. Note that $(2,p)$-generated groups form a large class of groups. For example, every finite nonabelian simple group is $(2,p)$-generated for some prime $p$~\cite{King2017}.

\begin{theorem}\label{thm1}
Let $k\geq5$ be an integer, and let $p\geq3\lceil k/2\rceil-2$ be a prime. Suppose that $G=\langle x,y\rangle$ is a finite group with $x^p=y^2=1$ and $yxy\notin\langle x\rangle$, and suppose that $G$ has no proper subgroup of index less than $4$. Let $R=\{x^{\pm1},x^{\pm2},\dots,x^{\pm\lfloor(k-1)/2\rfloor}\}$,
\[
S=
\begin{cases}
R\cup\{y\}&\textup{ if }k\textup{ is odd}\\
R\cup\{y,x^{-1}yx\}&\textup{ if }k\textup{ is even},
\end{cases}
\]
and $\Gamma=\Cay(G,S)$. Then $\Gamma$ is a GRR of $G$ if and only if $\Aut(G,S)=1$.
\end{theorem}

The proof of Theorem~\ref{thm1} will be given in Section~\ref{sec1}, where the idea of the key lemma (Lemma~\ref{lem7}) is from~\cite{Godsil1983}. In fact, our Theorem~\ref{thm1} is inspired by~\cite{Godsil1983}. Note that the Cayley graph $\Gamma$ in Theorem~\ref{thm1} has valency $k$. For the convenience in applying Theorem~\ref{thm1}, we introduce the following notation.

\begin{notation}
For an integer $k\geq5$, a group $G$, and elements $x$ and $y$ of $G$ with $|x|>2\lfloor(k-1)/2\rfloor$ and $|y|=2$, let
\[
\Gamma_k(G,x,y)=
\begin{cases}
\Cay\left(G,\{x^{\pm1},x^{\pm2},\dots,x^{\pm(k-1)/2},y\}\right)&\textup{ if }k\textup{ is odd}\\
\Cay\left(G,\{x^{\pm1},x^{\pm2},\dots,x^{\pm(k-2)/2},y,x^{-1}yx\}\right)&\textup{ if }k\textup{ is even}.
\end{cases}
\]
\end{notation}

As mentioned above, the finite nonabelian simple groups form a large class of $(2,p)$-generated groups, and the special interest in the existence of GRRs of nonabelian simple groups has been in the cubic case. For example, it was conjectured in~\cite{XF2016} and recently proved in~\cite{XZZ2022} that, except for a finite number of cases, every finite nonabelian simple group has a cubic GRR. Then a natural conjecture to extend this is:

\begin{conjecture}\label{conj1}
For each integer $k\geq3$, except for a finite number of cases, every finite nonabelian simple group has a $k$-valent GRR.
\end{conjecture}

Towards an affirmative solution to Conjecture~\ref{conj1}, we apply Theorem~\ref{thm1} in Sections~\ref{sec3} and~\ref{sec2} to study the existence of $k$-valent GRRs of alternating groups and groups of Lie type, respectively. Observe that, when applying Theorem~\ref{thm1} to nonabelian simple groups $G$, the conditions $yxy\notin\langle x\rangle$ and that $G$ has no proper subgroup of index less than $4$ are automatically satisfied. This reduces our task to proving the existence of elements $x$ and $y$ of order $p$ and $2$, respectively, such that $G=\langle x,y\rangle$ and $\Aut(G,S)=1$. For alternating groups, we have:

\begin{theorem}\label{thm3}
Let $k\geq5$ be an integer, let $G=\A_n$ with $n\geq\max\{14,6\lceil k/2\rceil-12\}$, and let $p$ be a prime such that $(n+4)/2<p\leq n-3$. Then for each element $x$ of order $p$ in $G$, there exists an involution $y$ in $G$ such that $\Gamma_k(G,x,y)$ is a GRR of $G$.
\end{theorem}

Section~\ref{sec3} is devoted to the proof of Theorem~\ref{thm3}. The proof is constructive, and it is interesting to remark that an ingredient in the proof is a celebrated theorem of Jordan in 1873 (see the proof of Lemma~\ref{lem4}). We also remark that, for an integer $n\geq14$, there does exist a prime $p$ with $(n+4)/2<p\leq n-3$ (see Lemma~\ref{lem2}). Thus the following corollary is an immediate consequence of Theorem~\ref{thm3}, which confirms Conjecture~\ref{conj1} for alternating groups when $k\geq5$.

\begin{corollary}\label{cor1}
For each integer $k\geq5$, there is a $k$-valent GRR of $\A_n$ for all $n\geq\max\{14,6\lceil k/2\rceil-12\}$.
\end{corollary}

For a power $a$ of a prime $r$, we call a prime $p$ a \emph{primitive prime divisor} of $a-1$ if $p$ divides $a-1$ but not $r^j-1$ for $j=1,\dots,\log_r(a)-1$. In other words, a primitive prime divisor of $a-1$ is a prime number $p$ such that $r$ has order $\log_r(a)$ in $\mathbb{F}_p^\times$. In particular, if $p$ is a primitive prime divisor of $a-1$ then $p\geq\log_r(a)+1$. As a consequence of Zsigmondy's theorem (see, for example,~\cite[Theorem IX.8.3]{Blackburn1982}), a primitive prime divisor of $r^m-1$ always exists for $m\geq7$. Based on Theorem~\ref{thm1}, we prove in Section~\ref{sec2} the following theorem, where a random involution is meant to be an involution chosen uniformly at random from all involutions of the group under consideration.

\begin{theorem}\label{thm2}
Let $k\geq5$ be an integer, let $G$ be a finite classical simple group, and let $p$ be a primitive prime divisor of $q^e-1$, where $G$ and $e$ are given in Table~$\ref{tab4}$ with prime power $q$. Suppose that $x$ is an element of order $p$ in $G$. Then for a random involution $y$ of $G$, the probability that $\Gamma_k(G,x,y)$ is a GRR of $G$ approaches $1$ as $q^n$ approaches infinity.
\end{theorem}

\begin{table}[htbp]
\caption{The classical group $G$ and number $e$ in Theorem~\ref{thm2}}\label{tab4}
\centering
\begin{tabular}{|l|l|l|l|}
\hline
row & $G$ & conditions & $e$\\
\hline
$1$ & $\PSL_n(q)$ & $n\geq\max\{9,(3k-3)/2\}$ & $n$\\
$2$ & $\PSU_n(q)$ & $n\geq\max\{5,(3k-3)/4\}$ odd & $2n$\\
$3$ & $\PSU_n(q)$ & $n\geq\max\{6,(3k+1)/4\}$ even & $2(n-1)$\\
$4$ & $\PSp_n(q)$ & $n\geq\max\{10,(3k-3)/2\}$ even & $n$\\
$5$ & $\POm_n(q)$ & $n\geq\max\{9,(3k-1)/2\}$ odd, $q$ odd & $n-1$\\
$6$ & $\POm_n^+(q)$ & $n\geq\max\{14,(3k+1)/2\}$ even & $n-2$\\
$7$ & $\POm_n^-(q)$ & $n\geq\max\{14,(3k-3)/2\}$ even & $n$\\
\hline
\end{tabular}
\end{table}

As a consequence of Theorem~\ref{thm2}, for each integer $k\geq5$, there are at most finitely many groups in Table~$\ref{tab4}$ that have no $k$-valent GRRs. Since there is no exceptional group of Lie type of rank larger than $8$, we then derive the following:

\begin{corollary}\label{cor2}
For each integer $k\geq5$, there are at most finitely many finite simple groups of Lie type of rank at least $\max\{9,(3k-6)/2\}$ that have no $k$-valent GRRs. In particular, for each integer $k\geq5$, there exists a constant $N(k)$ such that every finite simple group of Lie type of rank at least $N(k)$ has a $k$-valent GRR.
\end{corollary}

As a concluding remark we mention that, although Theorems~\ref{thm3} and~\ref{thm2} are on the existence of $k$-valent GRRs of nonabelian simple groups with $k\geq5$, similar techniques can be applied to treat some almost simple groups. (An almost simple group is by definition a group between the inner automorphism group and full automorphism group of some nonabelian simple group). For instance, with the same approach in Section~\ref{sec3} one may establish results on the symmetric group $\Sy_n$ that are similar to Theorem~\ref{thm3} and Corollary~\ref{cor1}.
However, constrained by Theorem~\ref{thm1}, our approach cannot be used to deal with the case when $k=4$. In this case, partial answers to Question~\ref{que1} for $S$ of size $4$ is desired.

\section{Preliminaries}

Hereafter, all groups are assumed to be finite, and all graphs are assumed finite and simple. We first give a necessary condition for Cayley graphs being GRRs as observed in the Introduction of~\cite{Godsil1983}.

\begin{lemma}\label{lem5}
Suppose that $\Cay(G,S)$ is a GRR of a group $G$ with $|G|\geq3$. Then there exists no nontrivial proper subgroup $H$ of $G$ such that $S\setminus H$ is a union of left cosets of $H$ in $G$. In particular, $G=\langle S\rangle$.
\end{lemma}

For a group $X$ acting on a set $\Omega$ and an element $v$ of $\Omega$, the stabilizer of $v$ in $X$ is denoted by $X_v$. For example, if $A$ is a subgroup of the automorphism group of a Cayley graph of $G$, then $A_1$ denotes the stabilizer in $A$ of the vertex $1\in G$.

\begin{lemma}\label{lem10}
Let $\Gamma$ be a Cayley graph of a group $G$, and let $A=\Aut(\Gamma)$. Suppose that $G$ is non-normal in $A$. Then $G$ has a proper subgroup of index less than $|A_1|$.
\end{lemma}

\begin{proof}
Let $\Omega$ be the set of right cosets of $G$ in $A$. Since $G$ is regular on the vertex set of $\Gamma$, we have $A=GA_1$ with $G\cap A_1=1$. Consequently, $|\Omega|=|A|/|G|=|A_1|$. Since $G$ is not normal in $A$, the action of $G$ by right multiplication on $\Omega$ is not trivial. Then since $G$ stabilizes $G\in\Omega$, it follows that $G$ has an orbit of length $k$ on $\Omega$ with $2\leq k\leq|\Omega|-1=|A_1|-1$. This implies that $G$ has a subgroup of index $k$ with $2\leq k\leq|A_1|-1$, which completes the proof.
\end{proof}

For a partition $\calB$ of the vertex set of a graph $\Gamma$, the \emph{quotient graph} $\Gamma_\calB$ with respect to $\calB$ is the graph with vertex set $\calB$ such that two vertices $B$ and $C$ of $\Gamma_\calB$ are adjacent if and only if $b$ is adjacent to $c$ in $\Gamma$ for some $b\in B$ and $c\in C$. If the partition $\calB$ is invariant under some $X\leq\Aut(\Gamma)$, then $X$ induces a subgroup of $\Aut(\Gamma_\calB)$. Among other conclusions, the next lemma shows the existence of nontrivial $\Aut(\Gamma)$-invariant partitions for certain Cayley graphs $\Gamma$. The lemma is folklore, but we give a full proof here for the reader's benefit.

\begin{lemma}\label{lem1}
Let $\Gamma$ be a Cayley graph of a group $G$, let $A=\Aut(\Gamma)$, and let $T\subseteq G$. Suppose that $A_1$ stabilizes $T$ setwise. Then the following statements hold:
\begin{enumerate}[{\rm (a)}]
\item $A$ is contained in the automorphism group of $\Cay(G,T\setminus\{1\})$;
\item $(Tg)^\alpha=Tg^\alpha$ for all $g\in G$ and $\alpha\in A_1$;
\item $A_1$ stabilizes $\langle T\rangle$ setwise;
\item the right cosets of $\langle T\rangle$ in $G$ form an $A$-invariant partition of $G$.
\end{enumerate}
\end{lemma}

\begin{proof}
Let $\Sigma=\Cay(G,T\setminus\{1\})$, let $R$ be the right regular permutation representation of $G$ and $\alpha$ be an arbitrary element of $A_1$. Note for any $x\in G$ that $R(x)\alpha R((x^\alpha)^{-1})$ is in $A_1$. Then for elements $x$ and $y$ in $G$, the condition $yx^{-1}\in T$ implies 
\[
y^\alpha(x^\alpha)^{-1}=(yx^{-1})^{R(x)\alpha R((x^\alpha)^{-1})}\in T,
\]
as $A_1$ stabilizes $T$. Hence $A_1\leq\Aut(\Sigma)$, which together with $R(G)\leq\Aut(\Sigma)$ leads to $A=R(G)A_1\leq\Aut(\Sigma)$, proving part~(a).

Note that $Tg\setminus\{g\}$ is the neighborhood of $g$ in $\Sigma$ and $Tg^\alpha\setminus\{g^\alpha\}$ is the neighborhood of $g^\alpha$ in $\Sigma$. We then deduce part~(b) from part~(a). Similarly, both parts~(c) and~(d) follow from part~(a), since the cosets of $\langle T\rangle$ in $G$ are the connected components of $\Sigma$.
\end{proof}

A Cayley graph $\Cay(G,S)$ is said to be \emph{normal} if $G$ is normal in $\Aut(\Cay(G,S))$. From~\eqref{eq2} we see that $\Cay(G,S)$ is normal if and only if $\Aut(\Cay(G,S))=G\rtimes\Aut(G,S)$. The following result is from~\cite[Example~2.2]{Xu1998}.

\begin{lemma}\label{lem3}
Let $p$ be an odd prime. Then every Cayley graph of $\C_p$ other than the null graph and the complete graph is normal.
\end{lemma}

For a group $X$ acting on a set $\Omega$, if $X$ stabilizes a subset $\Delta$ of $\Omega$ setwise, then the induced permutation group of $X$ on $\Delta$ will be denoted by $X^\Delta$. Let $\Gamma$ be a graph and $v$ be a vertex of $\Gamma$. Denote by $\Gamma_v^{[n]}$ the set of vertices of $\Gamma$ of distance at most $n$ to $v$ and denote by $\Gamma(v)=\Gamma_v^{[1]}\setminus\{v\}$ the neighborhood of $v$ in $\Gamma$. For $G\leq\Aut(\Gamma)$, denote by $G_v^{[n]}$ the pointwise stabilizer of $\Gamma_v^{[n]}$ in $G$. The proof of the next lemma is a standard argument (see, for example, the proof of Corollary (2) to Theorem~1 of~\cite{Neumann1973}).

\begin{lemma}\label{lem02}
Let $G\leq\Aut(\Gamma)$ be arc-transitive on a connected graph $\Gamma$ and $v$ be a vertex of $\Gamma$. If $G_v^{\Gamma(v)}$ is solvable then so is $G_v$.
\end{lemma}

The next lemma can be read off from the proof of the Satz in~\cite{Weiss1974}.

\begin{lemma}\label{lem04}
Let $G\leq\Aut(\Gamma)$ be arc-transitive on a connected graph $\Gamma$ and $\{u,v\}$ be an edge of $\Gamma$. If $|\Gamma(v)|\geq5$ is prime and $G_v$ is solvable, then $G_u^{[1]}\cap G_v^{[1]}=1$.
\end{lemma}

%
%
%
%

The last lemma in this section is a number-theoretic result that is slightly stronger than Bertrand's postulate.

\begin{lemma}\label{lem2}
For every integer $n\geq14$, there exists a prime $p$ such that $(n+4)/2<p\leq n-3$.
\end{lemma}

\begin{proof}
One may directly verifies the conclusion for $n\in\{14,\dots,31\}$. For $n\geq32$, since $5(n-2)/6\geq25$, we have by~\cite{Nagura1952} that there is a prime $p$ with $5(n-2)/6<p<n-2$. Such a prime satisfies $(n+4)/2<p\leq n-3$ as $5(n-2)/6>(n+4)/2$.
\end{proof}

\section{Cayley graphs of $(2,p)$-generated groups}\label{sec1}

The following result plays a key role in the proof of Theorem~\ref{thm1}. Although the proof of Theorem~\ref{thm1} only needs the result for $p\geq5$, we still include the case $p=3$ for its own interest.

\begin{lemma}\label{lem7}
Suppose that $G=\langle x,y\rangle$ with $x^p=y^2=1$ and $yxy\notin\langle x\rangle$, where $p$ is an odd prime, and suppose that $R$ is a nonempty inverse-closed subset of $\langle x\rangle\setminus\{1\}$ with $|\Aut(\langle x\rangle,R)|=2$. Let $S=R\cup M$ with
$M=\{y\}$ or $\{y,x^{-1}yx\}$, let $\Gamma=\Cay(G,S)$, and let $A=\Aut(\Gamma)$. If $A_1$ stabilizes $R$ setwise, then one of the following holds:
\begin{enumerate}[{\rm (a)}]
\item $p=3$, and $|A_1|$ divides $16$;
\item $p\geq5$, and $|A_1|$ divides $4$.
\end{enumerate}
\end{lemma}

\begin{proof}
Denote $H=\langle x\rangle$. We have $\langle R\rangle=H$ as $H$ is a cyclic group of prime order. Suppose that $A_1$ stabilizes $R$ setwise. Then by Lemma~\ref{lem1}, $A_1$ stabilizes $H$ setwise, and the right cosets of $H$ in $G$ form an $A$-invariant partition of $G$. Let $\Sigma$ be the quotient graph with respect to this partition, and $\overline{A}$ be the subgroup of $\Aut(\Sigma)$ induced by $A$. Then $\Sigma$ is connected since $\Gamma$ is connected, and $\overline{A}$ is vertex-transitive since $A$ is vertex-transitive.

Let $Hg$ be a neighbor of $H$ in $\Sigma$. Then there exist $h_1,h_2\in H$ such that $h_1$ is adjacent to $h_2g$ in $\Gamma$, which means $h_2gh_1^{-1}\in S=R\cup M$. If $h_2gh_1^{-1}\in R$, then 
\[
g\in h_2^{-1}Rh_1\subseteq h_2^{-1}Hh_1=H
\]
and so $Hg=H$, a contradiction. Hence $h_2gh_1^{-1}\in M\subseteq\{y,x^{-1}yx\}$. Consequently, $h_2gh_1^{-1}=y$ or $x^{-1}yx$, and so $Hg=Hyh_1$ or $Hyxh_1$. This shows that the neighbors of $H$ in $\Sigma$ have the form $Hyh$ with $h\in H$. Conversely, $Hyh$ is adjacent to $H$ in $\Sigma$ for each $h\in H$ because $yh$ is adjacent to $h$ in $\Gamma$. Therefore,
\begin{equation}\label{Eqn1}
\Sigma(H)=\{Hyh\mid h\in H\}.
\end{equation}
As $H$ acts transitively on $\{Hyh\mid h\in H\}$ by right multiplication and stabilizes $H$, we see that $\overline{A}$ is arc-transitive.

If $h_1,h_2\in H$ with $Hyh_1=Hyh_2$, then $yh_1h_2^{-1}y\in H$. Since $yxy\notin H$ and $H$ is cyclic of prime order, this implies that $h_1=h_2$. Consequently, the neighbors $Hyh$ of $H$ in $\Sigma$ are in a one-to-one correspondence with elements $h$ of $H$.
In particular, $\Sigma$ has valency $p$.

Consider a neighbor $Hyhg$ of $Hg$ in $\Sigma$, where $h\in H$ and $g\in G$. Suppose that $h_1g\in Hg$ and $h_2yhg\in Hyhg$ are two adjacent vertices in $\Gamma$. Then $h_2yhh_1^{-1}\in S=R\cup M$, and so we note from $y\notin H$ that $h_2yhh_1^{-1}\in M$.
Therefore, $h_2yhh_1^{-1}=y$ or $x^{-1}yx$, and so $yh_2y=h_1h^{-1}\in H$ or $yxh_2y=xh_1h^{-1}\in H$. Since $yxy\notin H$ and $H$ is a cyclic group of prime order, we deduce that either $h_2=1=h_1h^{-1}$ or $xh_2=1=xh_1h^{-1}$. This shows that there is exactly one edge $\{hg,yhg\}$ in $\Gamma$ joining the two sets $Hg$ and $Hyhg$ if $M=\{y\}$, and that there are exactly two edges $\{hg,yhg\}$ and $\{x^{-1}hg, x^{-1}yhg\}$ in $\Gamma$ joining the two sets $Hg$ and $Hyhg$ if $M=\{y,x^{-1}yx\}$.

We show in this paragraph that $A$ acts faithfully on $\Sigma$. Suppose that $\alpha\in A$ stabilizes each coset of $H$ in $G$ setwise.
By the conclusion of the previous paragraph, if $M=\{y\}$, then $\alpha$ fixes each vertex of $\Gamma$ lying on an edge joining two cosets of $H$, and so $\alpha$ fixes each vertex of $\Gamma$, which implies that $A$ acts faithfully on $\Sigma$. 
Now we assume $M=\{y,x^{-1}yx\}$. For $i\in\{0,1,\dots,p-1\}$, taking $h=x^i$ in the conclusion of the previous paragraph gives that $\{x^ig,yx^ig\}$ and $\{x^{i-1}g,x^{-1}yx^ig\}$ are the only two edges joining $Hg$ and $Hyx^ig$. Since $\alpha$ stabilizes each coset of $H$ in $G$ setwise, it follows that $\alpha$ stabilizes $\{x^ig,x^{i-1}g\}$ for all $i\in\{0,1,\dots,p-1\}$. As a consequence, $\alpha$ stabilizes $\{x^ig,x^{i-1}g\}\cap\{x^{i+1}g,x^ig\}=\{x^ig\}$ for all $i\in\{0,1,\dots,p-1\}$. This shows that $\alpha$ fixes each vertex of $\Gamma$, and so $A$ acts faithfully on $\Sigma$.

Let $B$ be the subgroup of $A$ stabilizing $H$ setwise and let $\overline{B}$ be the subgroup of $\Aut(\Sigma)$ induced by $B$. Then $A_1=B_1$, $\overline{B}\cong B$, and $\overline{B}$ is the stabilizer in $\overline{A}$ of the vertex $H$ of $\Sigma$. Since $H$ lies in $B$ and acts regularly on the set $H$ by right multiplication, we have
\begin{equation}\label{eq1}
|\overline{B}|=|B|=|H||B_1|=p|B_1|=p|A_1|.
\end{equation}
Recall that $\Sigma$ has valency $p$. If $p=3$, then by Tutte's theorem~\cite{Tutte1947}, $|\overline{B}|$ divides $48$, which together with~\eqref{eq1} implies that $|A_1|$ divides $16$, as part~(a) asserts. Thus we assume $p\geq5$ in the following.

If $R=H\setminus\{1\}$, then $\Aut(H,R)=\Aut(H)$ and so $|\Aut(H,R)|=p-1>2$, a contradiction. Consequently, $R\neq H\setminus\{1\}$ and hence $\Cay(H,R)$ is neither a null graph nor a complete graph. Note that $\Aut(H,R)$ consists of the identity map and the inverse map as $|\Aut(H,R)|=2$. Then by Lemma~\ref{lem3}, $\Aut(\Cay(H,R))=\C_p\rtimes\Aut(H,R)=\D_{2p}$. Since the induced subgraph of $H$ in $\Gamma$ is $\Cay(H,R)$, it follows that the induced group $B^H$ of $B$ on $H$ is a subgroup of $\D_{2p}$. In view of the one-to-one correspondence between $\Sigma(H)$ and $H$ we derive that $\overline{B}^{\Sigma(H)}$ is permutation isomorphic to $B^H$. Hence $\overline{B}^{\Sigma(H)}$ is a subgroup of $\D_{2p}$, and so by Lemma~\ref{lem02}, $\overline{B}$ is solvable.

Let $\overline{C}$ be the stabilizer in $\overline{A}$ of the vertex $Hy$ of $\Sigma$. According to Lemma~\ref{lem04} we have $\overline{B}^{[1]}\cap\overline{C}^{[1]}=1$, where $\overline{B}^{[1]}$ is the pointwise stabilizer of $\Sigma(H)$ in $\overline{B}$ and $\overline{C}^{[1]}$ is the pointwise stabilizer of $\Sigma(Hy)$ in $\overline{C}$. Note that $\overline{C}^{\Sigma(Hy)}$ is permutation isomorphic to $\overline{B}^{\Sigma(H)}$ and thus is permutation isomorphic to $B^H$. As $B^H\leq\D_{2p}$ and $\overline{B}^{[1]}\overline{C}^{[1]}/\overline{C}^{[1]}$ is a subgroup of $\overline{C}/\overline{C}^{[1]}=\overline{C}^{\Sigma(Hy)}$ stabilizing the neighbor $H$ of $Hy$ in $\Sigma$, we deduce that 
\[
|\overline{B}^{[1]}|=|\overline{B}^{[1]}|/|\overline{B}^{[1]}\cap\overline{C}^{[1]}|
=|\overline{B}^{[1]}\overline{C}^{[1]}/\overline{C}^{[1]}|\leq2
\]
and so $|\overline{B}|=|\overline{B}^{\Sigma(H)}||\overline{B}^{[1]}|=|B^H||\overline{B}^{[1]}|$ divides $2|\D_{2p}|=4p$. This together with~\eqref{eq1} shows that $|A_1|$ divides $4$, as part~(b) asserts.
\end{proof}

We will need the following lemma both in the proof and in the applications of Theorem~\ref{thm1}.

\begin{lemma}\label{lem6}
Let $m\geq5$ be an odd integer, let $\ell\geq(3m-1)/2$ be an integer, let $H=\langle x\rangle$ be a cyclic group of order $\ell$, and let $R=\{x^{\pm1},x^{\pm2},\dots,x^{\pm(m-1)/2}\}$. Then $|\Aut(H,R)|=2$.
\end{lemma}

\begin{proof}
Clearly, the identity map and the inverse map are in $\Aut(H,R)$. Suppose $\alpha\in\Aut(H,R)$ such that $x^\alpha=x^i$ with $2\leq|i|\leq(m-1)/2$. Let
\[
j=\left\lfloor\frac{m-1}{2|i|}\right\rfloor+1.
\]
Then we have
\begin{equation}\label{eq5}
\frac{m-1}{2|i|}<j\leq\frac{m-1}{2|i|}+1,
\end{equation}
which in conjunction with the assumption $2\leq|i|\leq(m-1)/2$ and $m\geq5$ yields that
\[
1<j\leq\frac{m-1}{4}+1\leq\frac{m-1}{2}.
\]
Therefore $x^j\in R$. However, since $|i|\leq(m-1)/2$ and $\ell\geq(3m-1)/2$, we derive from~\eqref{eq5} that
\[
\frac{m-1}{2}<|i|j\leq\frac{m-1}{2}+|i|\leq m-1\leq \ell-\frac{m+1}{2}<\ell-\frac{m-1}{2}.
\]
This implies that neither $x^{|i|j}$ nor $x^{-|i|j}$ is in $R$, and so $(x^j)^\alpha=x^{ij}\notin R$, contradicting $\alpha\in\Aut(H,R)$. Thus $\Aut(H,R)$ only contains the identity map and the inverse map, which means that $|\Aut(H,R)|=2$.
\end{proof}

A combination of Lemmas~\ref{lem7} and~\ref{lem6} leads to the next lemma.

\begin{lemma}\label{lem9}
Let $k\geq5$ be an integer and $p\geq3\lceil k/2\rceil-2$ be a prime. Suppose that $G=\langle x,y\rangle$ with $x^p=y^2=1$ and $yxy\notin\langle x\rangle$. Let $R=\{x^{\pm1},x^{\pm2},\dots,x^{\pm\lfloor(k-1)/2\rfloor}\}$,
\[
S=
\begin{cases}
R\cup\{y\}&\textup{ if }k\textup{ is odd}\\
R\cup\{y,x^{-1}yx\}&\textup{ if }k\textup{ is even},
\end{cases}
\]
$\Gamma=\Cay(G,S)$ and $A=\Aut(\Gamma)$. Then $|A_1|$ divides $4$.
\end{lemma}

\begin{proof}
Let $m=2\lfloor(k-1)/2\rfloor+1\geq5$. As $p\geq3\lceil k/2\rceil-2$, we have $p\geq(3m-1)/2$. Consider the induced subgraph $\Gamma[S]$ of $\Gamma$ on the neighborhood $S$ of $1$. Since $yxy\notin\langle x\rangle$, each vertex in $S\setminus R$ is isolated in $\Gamma[S]$. As $\{x^{\pm1},x^{\pm2}\}\subseteq R$, we conclude that $R$ is a connected component of $\Gamma[S]$, and so $A_1$ stabilizes $R$. Moreover, Lemma~\ref{lem6} asserts $|\Aut(\langle x\rangle,R)|=2$. Noting $p>3$, thus we derive from Lemma~\ref{lem7} that $|A_1|$ divides $4$.
\end{proof}

We are now ready to prove Theorem~\ref{thm1}.

\begin{proof}[Proof of Theorem~$\ref{thm1}$]
Let $A=\Aut(\Gamma)$. By Lemma~\ref{lem9} we have $|A_1|\leq4$. If $\Gamma$ is non-normal, then we deduce from Lemma~\ref{lem10} that $G$ has a proper subgroup of index less than $|A_1|$, contradicting the assumption that $G$ has no proper subgroup of index less than $4$. Thus $\Gamma$ is normal, and so~\eqref{eq2} leads to $\Aut(\Gamma)=G\rtimes\Aut(G,S)$. It follows that $\Aut(\Gamma)=G$ if and only if $\Aut(G,S)=1$. This completes the proof.
\end{proof}

To conclude this section, we give an observation that will be useful when applying Theorem~\ref{thm1}.

\begin{lemma}\label{lem14}
With the notation in Theorem~$\ref{thm1}$, if $\Aut(G,S)\neq1$, then there exists an involution $\alpha\in\Aut(G)$ such that $x^\alpha=x^{-1}$ and $y^\alpha=y$.
\end{lemma}

\begin{proof}
Suppose that there exists a nontrivial $\sigma\in\Aut(G,S)$. Since $R$ consists of the elements of order $p$ in $S$, it follows that $\sigma$ stabilizes $R$ and thus stabilizes $\langle x\rangle=\langle R\rangle$. Then Lemma~\ref{lem6} implies that $x^\sigma=x^{\pm1}$.

First assume that $\sigma$ stabilizes $R\cup\{y\}$. In this case, $\sigma$ fixes $y$ as it stabilizes $R$. If $x^\sigma=x$, then since $G=\langle x,y\rangle$, it follows that $\sigma$ fixes every element in $G$, contradicting $\sigma\neq1$. Hence $x^\sigma=x^{-1}$. Now $\sigma^2$ fixes both $x$ and $y$, and so $\sigma^2=1$. Then taking $\alpha=\sigma$ gives the conclusion of the lemma.

Next assume that $\sigma$ does not stabilize $R\cup\{y\}$. Then $k$ is even, and $\sigma$ swaps $y$ and $x^{-1}yx$. Let $\Inn(x)$ be the inner automorphism of $G$ induced by the conjugation by $x$, and let $\alpha=\sigma\Inn(x)^{-1}\in\Aut(G)$. We have
\[
x^\alpha=x(x^\sigma){x^{-1}}=x(x^{\pm1}){x^{-1}}=x^{\pm1}\ \text{ and }\ y^\alpha=x(y^\sigma){x^{-1}}=x(x^{-1}yx){x^{-1}}=y.
\]
It follows that $\alpha^2$ fixes both $x$ and $y$, and so $\alpha^2=1$ as $G=\langle x,y\rangle$. To complete the proof, we only need to show $\alpha\neq1$. Suppose for a contradiction that $\alpha=1$. Then $\Inn(x)=\sigma$ swaps $y$ and $x^{-1}yx$. In particular, $\Inn(x)$ maps $x^{-1}yx$ to $y$, that is, $x^{-2}y{x^2}=y$. Since $|x|=p$ is odd, this indicates that $x$ commutes with $y$, contradicting the condition $yxy\notin\langle x\rangle$.
\end{proof}

\section{GRRs of finite alternating groups}\label{sec3}

In this section we prove Theorem~\ref{thm3}.
In order to apply Theorem~\ref{thm1} to alternating groups, we construct the following elements $x$ and $y$ in $\A_n$. Recall from Lemma~\ref{lem2} that, for every integer $n\geq14$, there exists a prime $p$ with $(n+4)/2<p\leq n-3$.

\begin{construction}\label{cons1}
Let $G=\A_n$ with $n\geq14$, and let $p$ be a prime such that $(n+4)/2<p\leq n-3$. Take $x$ and $y$ in $G$ such that $x=(1,\dots,p)$ and
\[
y=
\begin{cases}
(p,p+1)\prod_{i=2}^{n-p}(i,i+p)&\textup{ if }n\textup{ is odd}\\
(p-2,p-1)(p,p+1)\prod_{i=2}^{n-p}(i,i+p)&\textup{ if }n\textup{ is even}.
\end{cases}
\]
\end{construction}

\begin{lemma}\label{lem4}
In the notation of Construction~$\ref{cons1}$, we have $G=\langle x,y\rangle$.
\end{lemma}

\begin{proof}
Since $x=(1,\dots,p)$, the orbit of $1$ under $\langle x\rangle$ is $\{1,\dots,p\}$. Note that, as $p>(n+4)/2$, we have $n-p<p$. For each $j\in\{p+1,\dots,n\}$, it follows from the definition of $y$ that $j^y\in\{1,\dots,p\}$. Hence $\langle x,y\rangle$ is transitive on $\{1,\dots,n\}$.
Suppose that $\langle x,y\rangle$ is imprimitive. Then there are integers $c>1$ and $d>1$ with $n=cd$ such that $\langle x,y\rangle$ is contained in $\Sy_c\wr\Sy_d$. As a consequence, $|\langle x,y\rangle|$ divides $(c!)^dd!$. However, $|\langle x,y\rangle|$ is divisible by $|x|=p$ while $p>n/2\geq\max\{c,d\}$, a contradiction.
Thus we conclude that $\langle x,y\rangle$ is primitive. Then since $x$ is a $p$-cycle with $p\leq n-3$, a theorem of Jordan (see~\cite[Theorem~3.3E]{DM1996} or~\cite[Theorem~13.9]{Wielandt1964}) asserts that $\langle x,y\rangle$ contains $\A_n$, and so $\langle x,y\rangle=G$.
\end{proof}

For a permutation $g$ of a set $\Omega$, let $\Fix(g)$ denote the set of points in $\Omega$ fixed by $g$.

\begin{lemma}\label{lem8}
In the notation of Construction~$\ref{cons1}$, if $p>2\lfloor(k-1)/2\rfloor$ for some integer $k\geq5$, then $\Aut(G,S)=1$ for the connection set $S$ of $\Gamma_k(G,x,y)$.
\end{lemma}

\begin{proof}
Let $S$ be the connection set of $\Gamma_k(G,x,y)$ with $k\geq5$ and $|x|=p>2\lfloor(k-1)/2\rfloor$.
By the definition of $x$ and $y$ we have
\begin{equation}\label{eq3}
\Fix(y)=
\begin{cases}
\{1\}\cup\{n-p+1,\dots,p-1\}&\textup{ if }n\textup{ is odd}\\
\{1\}\cup\{n-p+1,\dots,p-3\}&\textup{ if }n\textup{ is even}.
\end{cases}
\end{equation}
and
\begin{equation}\label{eq4}
\Fix(x^{-1}yx)=
\begin{cases}
\{2\}\cup\{n-p+2,\dots,p\}&\textup{ if }n\textup{ is odd}\\
\{2\}\cup\{n-p+2,\dots,p-2\}&\textup{ if }n\textup{ is even}.
\end{cases}
\end{equation}
Let $C$ be the $p$-cycle (as a graph) with vertices $1,\dots,p$ along the cycle. Then $\Fix(y)$ and $\Fix(x^{-1}yx)$ are subsets of $V(C)=\{1,\dots,p\}$. In the figures below, the blue indicates $\Fix(y)$ and the red indicates $\Fix(x^{-1}yx)$.

\begin{center}
\begin{tikzpicture}
\def\myds{1.5}
\draw (-3.5,0) coordinate (O);
\draw[thick] (O) circle[radius=1.5cm];
\draw (-3.5,1.5) coordinate (1) node[circle, inner sep=2.67pt, draw, fill=blue,label=above:$1$]{};
\node [circle, inner sep=2.67pt, draw, fill=red,label=above:$2$] (2) at ( (-3,1.4) {};
\node [circle, inner sep=2.67pt, draw, fill=red,label=above:$p$] (p) at ( (-4,1.4) {};
\node [circle, inner sep=2.67pt, draw, fill=blue,label=above left:$p-1$] (p-1) at ((2.56-7,1.16) {};
\node [circle, inner sep=2.67pt, draw, fill=blue,label=right:$n-p+1$] (n-p+1+7) at ((4.75-7,-0.82) {};
\node [circle, inner sep=2.67pt, draw, fill=red,label=below right:$n-p+2$] (n-p+2+7) at ((4.42-7,-1.168) {};
\node (p+e) at ((-3.725,1.375) {};
\node (p-1+e) at ((2.56-7+0.03,1.4) {};
\node (n-p+1+e) at ((4.97-7+0.02,-0.775) {};
\node (n-p+2+e) at ((4.48-7+0.06,-0.92) {};
\draw [line width=1.7pt,blue] (p-1+e) to[out=222,in=234,looseness=2.01] node [below left] {} (n-p+1+e);
\draw [line width=1.7pt,red] (p+e) to[out=205,in=211,looseness=2.1] node [above right] {} (n-p+2+e);
\node at  (-3.5,-2.7)  {\textsc{Figure 1.} $n$ odd};

\draw (3.5,0) coordinate (O);
\draw[thick] (O) circle[radius=1.5cm];
\draw (3.5,1.5) coordinate (1) node[circle, inner sep=2.67pt, draw, fill=blue,label=above:$1$]{};
\node [circle, inner sep=2.67pt, draw, fill=red,label=above:$2$] (2) at ( (4,1.4) {};
\node [circle, inner sep=2.67pt, draw, fill=red,label=left:$p-2$] (p-2) at ((2.245,0.795) {};
\node [circle, inner sep=2.67pt, draw, fill=blue,label=left:$p-3$] (p-3) at ((2.05,0.37) {};
\node [circle, inner sep=2.67pt, draw, fill=blue,label=right:$n-p+1$] (n-p+1+7) at ((4.75,-0.82) {};
\node [circle, inner sep=2.67pt, draw, fill=red,label=below right:$n-p+2$] (n-p+2+7) at ((4.42,-1.168) {};
\node (p-2+e) at ((2.46,0.88) {};
\node (p-3+e) at ((1.945,0.54) {};
\node (n-p+1+7+e) at ((4.97,-0.78) {};
\node (n-p+2+7+e) at ((4.48,-0.97) {};
\draw [line width=1.7pt,blue] (p-3+e) to[out=255,in=236,looseness=1.5] node [below left] {} (n-p+1+7+e);
\draw [line width=1.7pt,red] (p-2+e) to[out=239,in=214,looseness=1.45] node [above right] {} (n-p+2+7+e);
\node at  (3.5,-2.7)  {\textsc{Figure 2.} $n$ even};
\end{tikzpicture}
\end{center}

Take an arbitrary $\alpha\in\Aut(G,S)$. Then $\alpha$ is induced by the conjugation of some $g\in\Sy_n$. Let $R=\{x^{\pm1},x^{\pm2},\dots,x^{\pm\lfloor(k-1)/2\rfloor}\}$. Since $R$ consists of the elements of order $p$ in $S$, we see that $\alpha$ stabilizes $R$ and thus stabilizes $\langle R\rangle=\langle x\rangle$.
Since Lemma~\ref{lem6} implies that $g^{-1}xg=x^{\pm1}$, it follows that $g$ stabilizes $\{1,\dots,p\}=V(C)$, so $g$ induces a graph automorphism $\overline{g}$ of $C$. Recall that $S=R\cup\{y\}$ if $k$ is odd, and $S=R\cup\{y,x^{-1}yx\}$ if $k$ is even.

Suppose that $\alpha$ does not fix $y$. Then $k$ is even and $\alpha$ swaps $y$ and $x^{-1}yx$. It follows that $\alpha^2$ fixes both $x$ and $y$, and so $\alpha^2=1$ as $G=\langle x,y\rangle$. This implies $g^2=1$ and thus $\overline{g}^2=1$. Moreover, since $\alpha$ swaps $y$ and $x^{-1}yx$, it follows that $g$ swaps $\Fix(y)$ and $\Fix(x^{-1}yx)$.
Note that the condition $p>(n+4)/2$ indicates that $n-p+1<p-3$ and $n-p+2<p-2$. Then we see from~\eqref{eq3} and~\eqref{eq4} that $\overline{g}$ swaps $1$ and $2$, and so
\[
\overline{g}\colon i\mapsto((2-i)\bmod p)+1.
\]
However, in view of~\eqref{eq3} and~\eqref{eq4}, such an automorphism $\overline{g}$ of $C$ does not swap $\Fix(y)$ and $\Fix(x^{-1}yx)$, a contradiction.

Thus we conclude that $\alpha$ fixes $y$, that is, $g^{-1}yg=y$. In particular, $\Fix(y)$ is fixed by $g$ and hence fixed by the automorphism $\overline{g}$ of $C$. Then since $n-p+1<p-3$, we derive from~\eqref{eq3} that $\overline{g}=1$.
Therefore, $g^{-1}xg=x$, which together with $g^{-1}yg=y$ and $G=\langle x,y\rangle$ implies $\alpha=1$. This shows $\Aut(G,S)=1$, as required.
\end{proof}

\begin{proof}[Proof of Theorem~$\ref{thm3}$]
Since $n\geq6\lceil k/2\rceil-12$ and $p>(n+4)/2$, we have $p>3\lceil k/2\rceil-4$. As $3\lceil k/2\rceil-3$ is not prime, we then obtain $p\geq3\lceil k/2\rceil-2$. Since there is only one conjugacy class of elements of order $p$ in $G$, we may assume that $x=(1,\dots,p)$. Let $y$ be as in Construction~\ref{cons1}, and let $S$ be the connection set of $\Gamma_k(G,x,y)$. It follows from Lemmas~\ref{lem4} and~\ref{lem8} that $G=\langle x,y\rangle$ and $\Aut(G,S)=1$. Since $G=\langle x,y\rangle$ and $\langle x\rangle$ is not normal in $G=\A_n$, we derive that $yxy\notin\langle x\rangle$. Moreover, the smallest index of proper subgroups of $G=\A_n$ is $n\geq14$. Hence Theorem~\ref{thm1} asserts that $\Gamma_k(G,x,y)$ is a GRR of $G$.
\end{proof}


\section{GRRs of finite simple groups of Lie type}\label{sec2}

In this section, we apply Theorem~\ref{thm1} to study the existence of GRRs for some finite simple groups of Lie type, and establish Theorem~\ref{thm2}.
To keep the notation short, we identify a nonabelian simple group $G$ with its inner automorphism group, so that $G$ is viewed as a subgroup of $\Aut(G)$.
For a finite group $X$ denote the number of involutions in $X$ by $i_2(X)$. The following three lemmas are from~\cite[Proposition~3.1,~Lemma~4.2~and~Lemma~4.1]{Xia2020I}.

\begin{lemma}\label{lem13}
Let $G$ and $e$ be as in Table~$\ref{tab4}$, and let $p$ be a primitive prime divisor of $q^e-1$. Suppose that $x$ is an element of order $p$ in $G$.
Then for a random involution $y$ of $G$, the probability of $\langle x,y\rangle=G$ approaches $1$ as $q^n$ approaches infinity.
\end{lemma}

\begin{lemma}\label{lem12}
Let $G$ and $e$ be as in Table~$\ref{tab4}$, let $V$ be the natural module of the classical group $G$, and let $p$ be a primitive prime divisor of $q^e-1$. Suppose that $x\in G$ has order $p$ and $\alpha$ is an involution in $\Aut(G)\cap\PGL(V)$ such that $x^\alpha=x^{-1}$. Then $i_2(\Cen_G(\alpha))<m(G)i_2(G)$ with $m(G)$ given in Table~$\ref{tab15}$ corresponding to the same row of Table~$\ref{tab4}$, where $c$ is an absolute constant and $m(G)=0$ means that $\alpha$ does not exist for such $G$.
\end{lemma}

\begin{lemma}\label{lem11}
Let $G$ be a classical group in Table~$\ref{tab4}$, and let $V$ be the natural module of $G$. Suppose that $\alpha$ is an involution in $\Aut(G)\setminus\PGL(V)$. Then $i_2(\Cen_G(\alpha))<\ell(G)$ with $\ell(G)$ given in Table~$\ref{tab15}$ corresponding to the same row of Table~$\ref{tab4}$, where $\ell(G)=0$ means that $\alpha$ does not exist for such $G$.
\end{lemma}

\begin{table}[htbp]
\caption{The parameters of $G$ in the proof of Theorem~\ref{thm2}}\label{tab15}
\centering
\begin{tabular}{|l|l|l|l|l|l|l|l|}
\hline
row & $G$ & $m(G)$ & $\ell(G)$ & $i(G)$ & $|\Nor_G(\langle x\rangle)|$ & $u(G)$ & $v(G)$\\
\hline
\multirow{2}*{$1$} & \multirow{2}*{$\PSL_n(q)$} & \multirow{2}*{$cq^{-\frac{n^2}{10}}$} & \multirow{2}*{$3q^{\frac{n^2}{4}+\frac{n}{2}}$} & \multirow{2}*{$\frac{1}{8}q^{\lfloor\frac{n^2}{2}\rfloor}$} & \multirow{2}*{$\frac{n(q^n-1)}{(q-1)\gcd(n,q-1)}$} & \multirow{2}*{$Cq^{-\frac{91n}{90}+1}$} & \multirow{2}*{$8nq^{n-1}$} \\
 & & & & & & &\\
\multirow{2}*{$2$} & \multirow{2}*{$\PSU_n(q)$} & \multirow{2}*{$0$} & \multirow{2}*{$3q^{\frac{n^2}{4}+\frac{n}{4}}$} & \multirow{2}*{$\frac{1}{8}q^{\frac{n^2-1}{2}}$} & \multirow{2}*{$\frac{n(q^n+1)}{(q+1)\gcd(n,q+1)}$} & \multirow{2}*{$Cq^{-\frac{11n}{10}+1}$} & \multirow{2}*{$4nq^{n-1}$} \\
 & & & & & & &\\
\multirow{2}*{$3$} & \multirow{2}*{$\PSU_n(q)$} & \multirow{2}*{$0$} & \multirow{2}*{$3q^{\frac{n^2}{4}+\frac{n}{4}}$} & \multirow{2}*{$\frac{1}{8}q^{\frac{n^2}{2}}$} & \multirow{2}*{$\frac{(n-1)(q^{n-1}+1)}{\gcd(n,q+1)}$} & \multirow{2}*{$Cq^{-\frac{4n}{3}+1}$} & \multirow{2}*{$4nq^{n-1}$} \\
 & & & & & & &\\
\multirow{2}*{$4$} & \multirow{2}*{$\PSp_n(q)$} & \multirow{2}*{$cq^{-\frac{n^2}{20}-\frac{n}{10}}$} & \multirow{2}*{$3q^{\frac{n^2}{8}+\frac{n}{4}}$} & \multirow{2}*{$\frac{1}{2}q^{\frac{n^2}{4}+\frac{n}{2}}$} & \multirow{2}*{$\frac{n(q^\frac{n}{2}+1)}{\gcd(2,q-1)}$} & \multirow{2}*{$Cq^{-\frac{3n}{5}}$} & \multirow{2}*{$6nq^\frac{n}{2}$} \\
 & & & & & & &\\
\multirow{2}*{$5$} & \multirow{2}*{$\POm_n(q)$} & \multirow{2}*{$cq^{-\frac{n^2}{10}+\frac{1}{2}}$} & \multirow{2}*{$3q^{\frac{n^2-1}{8}}$} & \multirow{2}*{$\frac{1}{2}q^{\frac{n^2-1}{4}}$} & \multirow{2}*{$\frac{(n-1)(q^\frac{n-1}{2}+1)}{2}$} & \multirow{2}*{$Cq^{-\frac{9n}{10}+\frac{1}{2}}$} & \multirow{2}*{$4nq^{\frac{n}{2}-\frac{1}{2}}$} \\
 & & & & & & &\\
\multirow{2}*{$6$} & \multirow{2}*{$\POm_n^+(q)$} & \multirow{2}*{$cq^{-\frac{n^2}{20}+\frac{n}{10}}$} & \multirow{2}*{$3q^{\frac{n^2}{8}}$} & \multirow{2}*{$\frac{1}{8}q^{\frac{n^2}{4}-1}$} &\multirow{2}*{$\frac{2(n-2)(q^{\frac{n}{2}-1}+1)(q+1)}{\gcd(2,q-1)^2|\PSO_n^+(q)/G|}$} & \multirow{2}*{$Cq^{-\frac{3n}{5}}$} & \multirow{2}*{$14nq^\frac{n}{2}$} \\
 & & & & & & &\\
\multirow{2}*{$7$} & \multirow{2}*{$\POm_n^-(q)$} & \multirow{2}*{$cq^{-\frac{n^2}{20}+\frac{n}{10}}$} & \multirow{2}*{$0$} & \multirow{2}*{$\frac{1}{8}q^{\frac{n^2}{4}-1}$} &\multirow{2}*{$\frac{n(q^\frac{n}{2}+1)}{\gcd(2,q-1)|\PSO_n^-(q)/G|}$} & \multirow{2}*{$Cq^{-\frac{3n}{5}}$} & \multirow{2}*{$6nq^\frac{n}{2}$} \\
 & & & & & & &\\
\hline
\end{tabular}
\end{table}

\begin{proof}[Proof of Theorem~$\ref{thm2}$]
Let $V$ be the natural module of the classical group $G$, let
\[
R=\{x^{\pm1},x^{\pm2},\dots,x^{\pm\lfloor(k-1)/2\rfloor}\},
\]
and let $S$ be the connection set of $\Gamma_k(G,x,y)$, that is,
\[
S=
\begin{cases}
R\cup\{y\}&\textup{ if }k\textup{ is odd}\\
R\cup\{y,x^{-1}yx\}&\textup{ if }k\textup{ is even}.
\end{cases}
\]
Denote by $I$ and $J$ the sets of involutions of $G$ and $\Aut(G)$ respectively, and let
\[
K=\{y\in I\mid G=\langle x,y\rangle\}\quad\text{and}\quad L=\{y\in I\mid\Aut(G,S)=1\}.
\]
From the conditions in Table~\ref{tab4} we see that $e+1\geq(3k-1)/2$. Then since $p$ is a primitive prime divisor of $q^e-1$, it follows that $p\geq e+1\geq(3k-1)/2$. For $y\in I$, Theorem~\ref{thm1} asserts that $\Cay(G,S)$ is a GRR of $G$ if and only if $y\in K\cap L$, noting that $G$ has no proper subgroup of index less than $4$ as $G$ is simple.

From Lemma~\ref{lem14} we deduce that
\[
K\setminus L\subseteq\bigcup_{\substack{\alpha\in J\\ x^\alpha=x^{-1}}}(I\cap\Cen_G(\alpha)).
\]
By Lemmas~\ref{lem12}~and~\ref{lem11}, for each $\alpha\in J$ with $x^\alpha=x^{-1}$, we have
\[
\frac{i_2(\Cen_G(\alpha))}{i_2(G)}<\max\left\{m(G),\frac{\ell(G)}{i_2(G)}\right\}.
\]
Moreover, by~\cite[Proposition~3.1]{King2017}, $i_2(G)\geq i(G)$ with $i(G)$ in Table~$\ref{tab15}$ corresponding to the same row of Table~$\ref{tab4}$. We then conclude that
\[
\frac{i_2(\Cen_G(\alpha))}{i_2(G)}<\max\left\{m(G),\frac{\ell(G)}{i(G)}\right\}<u(G)
\]
with $u(G)$ in Table~\ref{tab15} corresponding to the same row of Table~$\ref{tab4}$, where $C$ is an absolute constant. Accordingly,
\[
|K\setminus L|\leq\sum_{\substack{\alpha\in J\\ x^\alpha=x^{-1}}}i_2(\Cen_G(\alpha))
<\sum_{\substack{\alpha\in J\\ x^\alpha=x^{-1}}}u(G)i_2(G)\leq u(G)i_2(G)|J\cap\Nor_{\Aut(G)}(\langle x\rangle)|.
\]
Since every involution in $\Nor_{\Aut(G)}(\langle x\rangle)$ projects to an involution or the identity in
\[
\Aut(G)/(\Aut(G)\cap\PGL(V)),
\]
we deduce that
\[
\frac{|J\cap\Nor_{\Aut(G)}(\langle x\rangle)|}{|(\Aut(G)\cap\PGL(V))\cap\Nor_{\Aut(G)}(\langle x\rangle)|}\leq i_2(\Aut(G)/(\Aut(G)\cap\PGL(V)))+1.
\]
Since $i_2(\Aut(G)/(\Aut(G)\cap\PGL(V)))\leq3$, we obtain
\[
\frac{|J\cap\Nor_{\Aut(G)}(\langle x\rangle)|}{|\Nor_{\Aut(G)\cap\PGL(V)}(\langle x\rangle)|}
=\frac{|J\cap\Nor_{\Aut(G)}(\langle x\rangle)|}{|(\Aut(G)\cap\PGL(V))\cap\Nor_{\Aut(G)}(\langle x\rangle)|}\leq4.
\]
It follows that
\[
|J\cap\Nor_{\Aut(G)}(\langle x\rangle)|\leq4|\Nor_{\Aut(G)\cap\PGL(V)}(\langle x\rangle)|
\leq\frac{4|\Nor_G(\langle x\rangle)||\Aut(G)\cap\PGL(V)|}{|G|}.
\]
As $|\Nor_G(\langle x\rangle)|$ is described in Table~\ref{tab15} (see~\cite[Proposition~6.4]{King2017}\setcounter{footnote}{0}\footnote{There is a factor $2$ of $|\Nor_G(\langle x\rangle)|$ missing for $G=\POm_n^+(q)$ in~\cite[Table~9]{King2017}, and this is corrected in our Table~2.}), this implies that
\[
|J\cap\Nor_{\Aut(G)}(\langle x\rangle)|<v(G)
\]
for $v(G)$ in Table~\ref{tab15} corresponding to the same row of Table~$\ref{tab4}$. Hence
\[
|K\setminus L|<u(G)i_2(G)|J\cap\Nor_{\Aut(G)}(\langle x\rangle)|<u(G)i_2(G)v(G),
\]
and so
\begin{equation}\label{eq6}
\frac{|K\cap L|}{|I|}=\frac{|K|}{|I|}-\frac{|K\setminus L|}{|I|}>\frac{|K|}{|I|}-u(G)v(G).
\end{equation}

By Lemma~\ref{lem13}, $|K|/|I|$ approaches $1$ as $q^n$ approaches infinity. In view of $n\leq\log_2(q^n)$, it is clear from Table~\ref{tab15} that $u(G)v(G)$ approaches $0$ as $q^n$ approaches infinity. Thus we conclude from~\eqref{eq6} that $|K\cap L|/|I|$ approaches $1$ as $q^n$ approaches infinity, which means that the probability that $\Gamma_k(G,x,y)$ is a GRR of $G$ approaches $1$ as $q^n$ approaches infinity.
\end{proof}

\vskip0.1in
\noindent\textbf{Acknowledgements.} The author would like to express his sincere gratitude to the anonymous referees for their careful reading and invaluable suggestions. The author also would like to thank Jack Moore for his comments on the Introduction and thank Wenying Zhu for her help during the preparation of this paper.

\end{document}